\documentclass[12pt]{amsart}
\usepackage{amsfonts, amsmath}

\usepackage{ amsmath, amsthm, amsfonts, amssymb, color}
 \usepackage{mathrsfs}
 \usepackage{amsmath,amstext,amsthm,amssymb,amsxtra}

\theoremstyle{plain}
\newtheorem{theorem}{Theorem}[section]

\newtheorem{proposition}[theorem]{Proposition}
\newtheorem{lemma}[theorem]{Lemma}

\theoremstyle{definition}

\newtheorem{remark}[theorem]{Remark}

\allowdisplaybreaks

\def\R{\mathbb R}

\newcommand\ZZ{\mathbb{Z}}

\def\11{1\!\!1}
\def\supp{\rm{supp}}
\def\SL{\sqrt{L}}

\def\l{\lambda}
\newcommand\D{\mathcal{D}}

\def\f12{\frac{1}{2}}

\DeclareMathOperator{\support}{supp}

\title[Weighted restriction type estimates for Grushin operators ]
{Weighted restriction type estimates for Grushin operators and application to spectral multipliers and Bochner-Riesz summability}

\author{Peng Chen}
\author{ El Maati Ouhabaz}

\address{Peng Chen,  Institut de Math\'ematiques de Bordeaux,  Univ.  Bordeaux, UMR 5251,
351, Cours de la Lib\'eration 33405 Talence, France}
\email{peng.chen@math.u-bordeaux1.fr}

\address{El Maati Ouhabaz,  Institut de Math\'ematiques de Bordeaux,  Univ.  Bordeaux, UMR CNRS 5251,
351, Cours de la Lib\'eration 33405 Talence, France}
\email{Elmaati.Ouhabaz@math.u-bordeaux1.fr}
\thanks{The research of both authors was partially supported by the ANR project  HAB, ANR-12-BS01-0013-02.}
\medskip
\subjclass[2000]{}
\date{\today}

\begin{document}

\begin{abstract} We prove weighted restriction type estimates for Grushin operators. These estimates are then used to prove
 sharp  spectral multiplier theorems  as well as  Bochner-Riesz summability results  with sharp exponent.
\end{abstract}

\maketitle

% \tableofcontents

\bigskip

   \section{Introduction  }
\setcounter{equation}{0}
We consider Grushin operators  on $\R^{d_1}\times \R^{d_2} = \R_{x'}^{d_1}\times \R_{x''}^{d_2}$ defined by
\begin{eqnarray}\label{def:Grushin}
L :=-\sum_{j=1}^{d_1}\partial_{x'_j}^2-\big(\sum_{j=1}^{d_1}
|x'_j|^2\big)\sum_{k=1}^{d_2}\partial_{x''_k}^2.
\end{eqnarray}
Such operators, defined by the quadratic form technique, are self-adjoint in $L^2(\R^{d_1}\times \R^{d_2})$.
Let $E_{L}(\lambda)$ be the spectral resolution of the operator~$L$ for  $\lambda \ge 0$. By the spectral theorem
for every bounded Borel function $F : \R \to \mathbb{C}$, one can define
\begin{equation}\label{el}
F(L) = \int_0^\infty  F(\lambda) \,dE_{L}(\lambda).
\end{equation}
 The operator $F(L)$ is bounded on $L^2(\R^{d_1}\times\R^{d_2})$. This paper is devoted to  spectral
multiplier results  for $L$,  that is, we investigate minimal sufficient condition on  $F$ under
which the operator $F(L)$ extends to a bounded operator on
$L^p(\R^{d_1}\times\R^{d_2})$ for some range of $p$. In this context, the minimal condition on $F$ we have in mind is the same as
in the Fourier multiplier theorem, i.e., boundedness of $F(-\Delta)$ on $L^p(\R^d)$ where $\Delta$ is the Euclidean Laplacian.
We also study the closely related question of critical exponent $\delta$ for which the Bochner-Riesz means $(1-t
L)_+^\delta$ are bounded on $L^p(\R^{d_1}\times\R^{d_2})$ uniformly in $t \in [0,\infty)$.

Spectral multipliers and Bochner-Riesz summability for Grushin operators have been studied recently by other authors. In  \cite{MM},
it is  proved that for  $\delta> \frac{1}{2}(d_1+d_2)-\frac{1}{2}$, the Bochner-Riesz means $(1-t
L)_+^\delta$ are bounded on $L^p(\R^{d_1}\times\R^{d_2})$ uniformly in $t \in [0,\infty)$ for all $1\leq p\leq \infty$.  A previous result was proved in \cite{MS}  with the condition $\delta > \frac{1}{2}\max(d_1 + d_2, 2d_2) -\frac{1}{2}$. 
Our aim is to get  similar
 results for smaller values of $\delta$, i.e. when  $0<\delta<\frac{1}{2}(d_1+d_2)-\frac{1}{2}$. In this case, we cannot   hope for  $(1-t
L)_+^\delta$ to be bounded on $L^p(\R^{d_1}\times\R^{d_2})$ for all  $p \in [1,  \infty]$. Our aim is to prove that  $(1-t
L)_+^\delta$ are bounded on $L^p(\R^{d_1}\times\R^{d_2})$ uniformly in $t$ for $p$  in some symmetric interval $[p_\delta, p_\delta']$ around $2$. The value  $p_\delta$ depends of course on  $\delta$. Such questions have been studied for the Euclidean Laplacian in which case the optimality of $\delta$ is known but the optimality of  $p$ is a  celebrate open problem, known as the Bochner-Riesz problem. See  \cite{St2}, p. 420 and \cite{T1} for more details and recent progress on this problem.

Starting from the result quoted above from \cite{MM} and  \cite{MS},   one  can use complex interpolation between $L^2$ boundedness for  any $\delta>0$ and $L^1$ boundedness for a fixed  $\delta >(d_1+d_2)/2-1/2$ to obtain  that for  $\delta>(d_1+d_2-1)|1/p-1/2|$, $(1-tL)_+^\delta$ are bounded on $L^p(\R^{d_1}\times\R^{d_2})$ uniformly in $t$. Note however that this strategy does not give the  optimal exponent. For example, when $L=-\Delta$  on $\R^n$, $(1+t\Delta)_+^\delta$ are bounded on $L^p(\R^{n})$ uniformly   when $\delta>\max\{n|1/p-1/2|-1/2,0\}$ for $1\leq p\leq (2n+2)/(n+3)$, which is better than the interpolation approach   which leads to $\delta>(n-1)|1/p-1/2|$.
The  sharpened  result for the Laplacian, i.e.,  $\delta>\max\{n|1/p-1/2|-1/2,0\}$ for $1\leq p\leq (2n+2)/(n+3)$,  is obtained by the restriction theorem for the Fourier transform on the unit sphere.
 In an abstract setting, versions of the restriction estimate are introduced  in \cite{COSY} and we are tempted 
  to follow  \cite{COSY}  in order to prove boundedness of Bochner-Riesz means for $L$. There is however an obstacle. The restriction type estimate introduced in \cite{COSY}
leads  to spectral multipliers using ``the"  homogeneous dimension $Q=d_1+2d_2$  rather than the topological one $d_1 + d_2$.
The   exponent we will get for the Bochner-Riesz means  is then $\max\{Q|1/p-1/2|-1/2,0\}$.  The problem of getting sharp spectral multipliers using  the topological dimension rather than the homogeneous one appeared already in the case of the Heisenberg group. See \cite{Heb} and \cite{MuSt}.

Our strategy to deal with this problem is  to use a   weighted  version of restriction  estimates  for the operator $L$. More precisely,
let $F$ be a bounded Borel function with support  $\support F$ contained in $[R/4,R]$ for some $R>0$. Then for $1\leq p\leq \min\{2d_1/(d_1+2),(2d_2+2)/(d_2+3)\}$ and $0\leq \gamma< d_2(1/p-1/2)$, we prove that
\begin{eqnarray*}
\||x'|^{\gamma}F(\sqrt{L})f\|_{L^2(\R^{d_1}\times \R^{d_2})}
\leq C R^{(2d_2+d_1)(1/p-1/2)-\gamma}\|\delta_R F\|_{L^2(\R)}\|f\|_{L^p(\R^{d_1}\times \R^{d_2})}.
\end{eqnarray*}
Using this weighted restriction type estimate, we prove  sharp spectral multiplier results  and optimal Bochner-Riesz summability
stated in Theorems~\ref{thm1m} and \ref{thm2r} below. We set
 $$D :=\max\{d_1+d_2,2d_2\}$$
 and denote as usual  $W_2^s$ the  $L^2$ Sobolev space of order $s$ with
$ \| F \|_{W_2^s} := \| (I-d_x^2)^{s/2}F  \|_2$. Throughout,   $\eta$ is an auxiliary and non trivial   $C^\infty$ function with compact support contained in $(0, \infty)$.

\begin{theorem}\label{thm1m}
Let  $1\leq p\leq \min\{2d_1/(d_1+2),(2d_2+2)/(d_2+3)\}$. Suppose that the bounded Borel function $F : \R \to \mathbb{C}$ satisfies
\begin{equation*}
\sup_{t>0} \|\eta \,  F(t\cdot)  \|_{W_2^s}  < \infty
\end{equation*}
for some $s >\max\{ D|1/p-1/2|,1/2\}$. Then the spectral multiplier operator $F(L)$ is bounded on
 $L^p(\R^{d_1}\times\R^{d_2})$. In addition
\begin{equation*}
 \|F(L)\|_{L^p \to L^{p}} \leq C_p \sup_{t>0} \|\eta \,  F (t\cdot) \|_{W_2^s}.
\end{equation*}

\end{theorem}

For   Bochner-Riesz means we prove
the following result.

\begin{theorem}\label{thm2r}
Let  $1\leq p\leq \min\{2d_1/(d_1+2),(2d_2+2)/(d_2+3)\}$. Suppose that $\delta > \max\{D|1/p-1/2|-1/2,0\}$. Then the Bochner-Riesz means $(1-t
L)_+^\delta$ are bounded on $L^p(\R^{d_1}\times\R^{d_2})$ uniformly in $t \in [
0,\infty)$.
\end{theorem}

Theorems~\ref{thm1m} and \ref{thm2r} are optimal  when $d_1\ge d_2$. In this case $D$  coincides with the topological dimension $d_1+d_2$ of $\R^{d_1}\times \R^{d_2}$.
By the elliptic property of $L$ in the region where $x'\neq 0$, one can use the transplantation argument described in \cite{KST} to deduce the sharpness of the above theorems from the fact that the exponent $D|1/p-1/2|-1/2$ is sharp for the classical Bochner-Riesz summability on $\R^D$. See also \cite{MS} and \cite{MM}.

\medskip

\noindent{\bf Conjecture.} We believe that the previous theorems are true with $D=d_1 + d_2$ instead of $D = \max(d_1 + d_2, 2d_2)$. As we mentioned above, if $p = 1$, the spectral multiplier theorem in  \cite{MM} is valid for  $s > \frac{1}{2}(d_1 + d_2)$. This means that the conjecture is true when $p = 1$.

\bigskip

Throughout, the symbols ``$c$" and ``$C$" will denote (possibly
different) positive constants that are independent of the essential
variables. The notation $A \sim B$ means that the quantities $A$ and $B$ satisfy $c A \le B \le C A$
for some positive constants $c$ and $C$.

\medskip
\section{Riemannian distance and the heat kernel estimates}
\label{sec:HeatKernelE} \setcounter{equation}{0}

Heat kernel bounds for  Grushin type operators have been proved in \cite{RS}.   Here we state  some  basic
results concerning the Riemannian distance associated with the Grushin  operator $L$
and recall the  Gaussian bound for the corresponding heat kernel.   

Recall that the Riemannian (quasi-)distance corresponding to the operator
$L$  can be defined by
\begin{equation*}
\rho(x,y)=\sup_{\psi\in \mathbb{D}}\,(\psi(x)-\psi(y))
\end{equation*}
for all $x = (x',x'') ,y = (y',y'') \in\R^{d_1}\times\R^{d_2}$
where
 \[
\mathbb{D}=\left\{\psi\in W^{1,\infty}(\R^{d_1}\times\R^{d_2}):\Big(\sum_{{{j}}=1}^{d_1}|\partial_{x'_{{j}}}\psi|^2 +
\Big(\sum_{{{j}}=1}^{d_1}|x'_{{j}}|^2 \Big)
\sum_{{{k}}=1}^{d_2}|\partial_{x''_{{k}}}\psi|^2\Big) \leq1\right\}
\;.
\]
For this distance $\rho$ and the  Lebesgue measure the  finite speed propagation property for the corresponding wave equation as well as
Gaussian estimates for the heat kernel of $L$ are satisfied. See \cite[Proposition 4.1]{RS} for more detailed
 discussion and references.

\begin{theorem}\label{th:Gaussian}
Let  $\rho$ be Riemannian distance associated with   the Grushin
operator $L$. Then for $x=(x',x'')$, $y=(y',y'') \in \R^{d_1}\times
\R^{d_2}$,
\begin{equation}\label{eqCD}
\rho(x,y) \sim |x' - y'| + \begin{cases}
\frac{|x''-y''|}{|x'| + |y'|} &\text{if $|x''-y''|^{1/2} \leq |x'| + |y'|$,}\\
|x''-y''|^{1/2} &\text{if $|x''-y''|^{1/2} \geq |x'| + |y'|$.}
\end{cases}
\end{equation}
Moreover the volume of the ball $B(x,r):=\{y\in \R^{d_1}\times \R^{d_2}:\, \rho(x,y)<r\}$ satisfies the following estimates
\begin{equation}\label{eqVr}
|B(x,r)| \sim r^{d_1+d_2} \max\{r,|x'|\}^{d_2},
\end{equation}
and in particular, for all $\lambda \geq 0$,
\begin{equation}\label{doubling}
|B(x,\lambda r)| \leq C (1+\lambda)^Q |B(x,r)|
\end{equation}
where $Q=d_1+2d_2$ is ``the''  homogenous dimension of the considered metric space.
Next, there exist constants $b,C > 0$ such that, for all $t > 0$, the integral kernel $p_t$ of the operator $\exp(-tL)$ satisfies  the following Gaussian bound
\begin{equation}\label{Gauss}
|p_t(x,y)| \leq C |B(y,t^{1/2})|^{-1} e^{-b \rho(x,y)^2/t}
\end{equation}
for all $x,y \in \R^{d_1}\times\R^{d_2}$.
\end{theorem}
\begin{proof} For the proof, we refer the reader to
\cite[Proposition 5.1 and Corollary 6.6]{RS}.
\end{proof}

\medskip

\section{Weighted restriction estimates}
\label{sec:WeightedRE} \setcounter{equation}{0}

In this section, we discuss the spectral decomposition of $L$  and then state and prove the weighted restriction estimate. 
\medskip

Let
$\mathcal{F} : L^2(\R^{d_1} \times \R^{d_2}) \to L^2(\R^{d_1} \times
\R^{d_2})$ be the partial Fourier transform in the variable $x''$, that is
\[
\mathcal{F} \phi(x',\xi)=\widehat{\phi}(x',\xi) = (2\pi)^{-d_2/2} \int_{\R^{d_2}}
\phi(x',x'') \, e^{-i \xi \cdot x''} \, dx''.
\]
 Then
\begin{equation}\label{eq03}
\mathcal{F} L\phi (x',\xi) = {L}_{\xi} \, \mathcal{F}
\phi(x',\xi),
\end{equation}
where ${L}_{\xi}$ is the Schr\"odinger  operator defined by
$$
{L}_{\xi}=-\Delta_{d_1}+|x'|^2|\xi|^2
$$
acting on $L^2(\R^{d_1})$ where $\xi\in \R^{d_2}$.  We  have the following proposition.
\begin{proposition}\label{prop:FL}
For any integrable function $F$ with compact support in $\R$, we have
$$
F(L)f(x',x'')=\mathcal{F}^{-1}(F(L_\xi)\widehat{f}(x',\xi))(x'').
$$
\end{proposition}
\begin{proof}
This equality is  essentially proved in \cite[Proposition 5]{MS}. Alternatively, we can follow  the approach used in the proof of Proposition 3.2 in \cite{CS} for a direct proof.
\end{proof}

Next we turn to the spectral decomposition of the operator $L_\xi$ on $\R^{d_1}$. Let $L_1=-\Delta_{d_1}+|x'|^2$ be the harmonic oscillator on $\R^{d_1}$, $\nu$ be a multi-index and $\Phi_\nu(x')=h_{\nu_1}(x'_1)\cdots h_{\nu_{d_1}}(x'_{d_1})$, where $h_{\nu_j}$ is the Hermite function of order~$\nu_j$. Recall that $2|\nu|+d_1$ and  $\Phi_\nu$ are the eigenvalues and eigenfunctions of the operator $L_1$. Thus  $(2|\nu|+d_1)|\xi|$ and  $\Phi^\xi_\nu(x')=|\xi|^{d_1/4}\Phi_{\nu}(\sqrt{|\xi|}x')$ are the eigenvalues and eigenfunctions of the operator $L_\xi$; see \cite{MS}. Then we have
$$
L_\xi f=\sum_{k=0}^{\infty} (2k+d_1)|\xi|\sum_{|\nu|=k}\langle f,\Phi_\nu^\xi \rangle \Phi_\nu^\xi
$$
and
$$
F(L_\xi) f=\sum_{k=0}^{\infty} F((2k+d_1)|\xi|)\sum_{|\nu|=k}\langle f,\Phi_\nu^\xi \rangle \Phi_\nu^\xi.
$$

We have the following restriction type estimate for $L_\xi$.

\begin{proposition}\label{prop:REforLxi}
Suppose $d_1\geq 2$. For $1\leq p\leq 2d_1/(d_1+2)$,
\begin{eqnarray}\label{eq:ReforLxi}
\|\sum_{|\nu|=k}\langle f,\Phi_\nu^\xi \rangle \Phi_\nu^\xi\|_{L^2(\R^{d_1})}\leq C |\xi|^{\frac{d_1}{2}(\frac{1}{p}-\frac{1}{2})}(2k+d_1)^{\frac{d_1}{2}(\frac{1}{p}-\frac{1}{2})-\frac{1}{2}}\|f\|_{L^p(\R^{d_1})}.
\end{eqnarray}
\end{proposition}
\begin{proof}
From \cite[Corollary 3.2]{KT} we have for $1\leq p\leq 2d_1/(d_1+2)$,
$$
\|\sum_{|\nu|=k}\langle f,\Phi_\nu \rangle \Phi_\nu \|_{L^2}\leq C
(2k+d_1)^{\frac{d_1}{2}(\frac{1}{p}-\frac{1}{2})-\frac{1}{2}} \|f\|_{L^p}.
$$
Then changing variables from the above inequality implies
\begin{eqnarray*}
\|\sum_{|\nu|=k}\langle f,\Phi_\nu^\xi \rangle \Phi_\nu^\xi\|_{L^2}
&=& \|\sum_{|\nu|=k}\left\langle f,|\xi|^{d_1/4}\Phi_{\nu}(\sqrt{|\xi|}\cdot) \right\rangle |\xi|^{d_1/4}\Phi_{\nu}(\sqrt{|\xi|}x')\|_{L^2}\\
&=& \|\sum_{|\nu|=k}\left\langle f(\frac{\cdot}{\sqrt{|\xi|}}),\Phi_{\nu}(\cdot) \right\rangle \Phi_{\nu}(\sqrt{|\xi|}x')\|_{L^2}\\
&=& \|\sum_{|\nu|=k}\left\langle f(\frac{\cdot}{\sqrt{|\xi|}}),\Phi_{\nu}(\cdot) \right\rangle \Phi_{\nu}(x')\|_{L^2}|\xi|^{-d_1/4}\\
&\leq& C |\xi|^{-d_1/4} (2k+d_1)^{\frac{d_1}{2}(\frac{1}{p}-\frac{1}{2})-\frac{1}{2}} \|f(\frac{\cdot}{\sqrt{|\xi|}})\|_{L^p}\\
&\leq&  C |\xi|^{\frac{d_1}{2}(\frac{1}{p}-\frac{1}{2})}(2k+d_1)^{\frac{d_1}{2}(\frac{1}{p}-\frac{1}{2})-\frac{1}{2}}\|f\|_{L^p(\R^{d_1})}.
\end{eqnarray*}
\end{proof}

In order to prove the weighted restriction  estimate  for $L$, we need the following proposition, which is essentially the same as \cite[Proposition 4]{MS}.

\begin{proposition}\label{prop:WeightedEforLxi}
Let $\gamma\in [0,\infty)$ and $f\in L^2(\R^{d_1})$. Then
$$
\||x'|^\gamma f\|_{L^2(\R^{d_1})}\leq C_\gamma\||\xi|^{-\gamma}L_\xi^{\gamma/2} f\|_{L^2(\R^{d_1})}.
$$
Here $C_\gamma$ is  non-decreasing in  $\gamma$.
\end{proposition}
\begin{proof}
Let $H$ be the harmonic oscillator $-d^2/du^2+u^2$ on $\R$. It is obvious that
$$
\||u|f\|_{L^2(\R)}\leq  \|H^{1/2}f\|_{L^2(\R)}.
$$
In addition since the first eigenvalue of $H$ is bigger than 1, 
\begin{eqnarray*}
\|\frac{d^2}{du^2}f\|_2^2 + \|u^2 f\|_2^2 &\leq&
\|(-\frac{d^2}{du^2}+u^2)f\|_2^2-2Re\langle
-\frac{d^2}{du^2}f, u^2f\rangle\\
&\leq & \|Hf\|_2^2-2Re\langle \frac{d}{du}f,
2uf\rangle-2\|u\frac{d}{du}f\|_2^2\\
&\leq & \|Hf\|_2^2-2Re\langle \frac{d}{du}f, 2uf\rangle\\
&\leq & \|Hf\|_2^2+4\|\frac{d}{du}f\|_2\|uf\|_2\\
&\leq & \|Hf\|_2^2+4\|H^{1/2}f\|_2\|H^{1/2}f\|_2\\
&\leq &5 \|Hf\|_2^2.
\end{eqnarray*}
This implies  that
$$
\|u^2f\|_{L^2(\R)}\leq \sqrt{5} \|Hf\|_{L^2(\R)}.
$$
By iteration, we can prove that for $k\in\mathbb{N}$,
$$
\|u^kf\|_{L^2(\R)}\leq C_k\|H^{k/2}f\|_{L^2(\R)}.
$$
For  details, we refer  to Proposition 3.2 and 3.3 in \cite{G}. Now by a similar approach as in  in the proof of Proposition 2.2 in \cite{CS}, we can prove Proposition~\ref{prop:WeightedEforLxi}.
\end{proof}

We  state our weighted restriction  estimate for the Grushin operator $L$.

\begin{theorem}\label{thm:WeightedRE}
Let $F$ be a Borel function with $\support F\subset [R/4,R]$ for some $R>0$. Then for $1\leq p\leq \min\{2d_1/(d_1+2),(2d_2+2)/(d_2+3)\}$ and
$0\leq \gamma<d_2(1/p-1/2)$,
\begin{eqnarray}\label{eq:WeightedRE}
\||x'|^{\gamma}F(\sqrt{L})f\|_{L^2(\R_{x'}^{d_1}\times \R_{x''}^{d_2})}
\leq C R^{(2d_2+d_1)(\frac{1}{p}-\frac{1}{2})-\gamma}\|\delta_R F\|_{L^2(\R)}\|f\|_{L^p(\R^{d_1}\times \R^{d_2})}.
\end{eqnarray}
Moreover,  when $|y'|>4r$,
\begin{eqnarray}\label{eq:WeightedRE2}
&&\||x'|^\gamma F(\sqrt{L})P_{B(y,r)}f\|_{L^2(\R_{x'}^{d_1}\times
\R_{x''}^{d_2})} \leq  C
R^{(d_2+d_1)(\frac{1}{p}-\frac{1}{2})}|y'|^{\gamma-d_2(\frac{1}{p}-\frac{1}{2})}\|\delta_R
F\|_{L^2} \|f\|_{L^p(\R^{d_1}\times\R^{d_2})},
\end{eqnarray}
where $y = (y', y'') \in \R^{d_1}\times\R^{d_2}$ and $P_{B(y,r)}$ is the projection on the ball $B(y,r)$ of  $\R^{d_1}\times\R^{d_2}$
for distance $\rho$.
\end{theorem}

\begin{proof}
When $p=1$, this theorem is proved in \cite[Proposition 10]{MS}. So in what follows we may  assume that $d_1 > 2$.
Let $G(x)=F(\sqrt{x})$. Then $\support G\subset [R^2/16,R^2]$.
By a density argument, it  is enough  to prove the estimates~\eqref{eq:WeightedRE} and \eqref{eq:WeightedRE2} for functions $f\in L^2(\R^{d_1}\times \R^{d_2})\cap L^p(\R^{d_1}\times \R^{d_2})$ such that $f(x',x'')=g(x')h(x'')$ where $g\in L^2(\R^{d_1})\cap L^p(\R^{d_1}) $
and $h\in L^2(\R^{d_2})\cap L^p(\R^{d_2})$.

By Proposition~\ref{prop:FL} and Plancherel equality,
\begin{eqnarray}\label{eq:We1}
\||x'|^\gamma F(\sqrt{L})f\|_{L^2(\R_{x'}^{d_1}\times \R_{x''}^{d_2})}^2 & = & \|\mathcal{F}^{-1}(|x'|^\gamma G(L_\xi)g(x')\widehat{h}(\xi))(x'')\|_{L^2(\R_{x'}^{d_1}\times \R_{x''}^{d_2})}^2\nonumber\\
& = & \||x'|^\gamma G(L_\xi)g(x')\widehat{h}(\xi)\|_{L^2(\R_{x'}^{d_1}\times \R_{\xi}^{d_2})}^2.
\end{eqnarray}
Then by Proposition~\ref{prop:WeightedEforLxi},
\begin{eqnarray*}
&&\||x'|^\gamma G(L_\xi)g(x')\widehat{h}(\xi)\|_{L^2(\R_{x'}^{d_1})}^2\\
&& \leq  \||\xi|^{-\gamma}L_\xi^{\gamma/2} G(L_\xi)g(x')\widehat{h}(\xi)\|_{L^2(\R_{x'}^{d_1})}^2\\
&&=  \||\xi|^{-\gamma}\sum_{k=0}^\infty ((2k+d_1)|\xi|)^{\gamma/2} G((2k+d_1)|\xi|)\sum_{|\nu|=k}\langle g, \Phi_\nu^\xi\rangle \Phi_\nu^\xi(x')\widehat{h}(\xi)\|_{L^2(\R_{x'}^{d_1})}^2,
\end{eqnarray*}
and by the orthonormal property for eigenfunctions of different eigenvalues, we have
\begin{eqnarray*}
&&\||x'|^\gamma G(L_\xi)g(x')\widehat{h}(\xi)\|_{L^2(\R_{x'}^{d_1})}^2\\
&& \leq  \sum_{k=0}^\infty\||\xi|^{-\gamma} ((2k+d_1)|\xi|)^{\gamma/2} G((2k+d_1)|\xi|)\sum_{|\nu|=k}\langle g, \Phi_\nu^\xi\rangle \Phi_\nu^\xi(x')\widehat{h}(\xi)\|_{L^2(\R_{x'}^{d_1})}^2.
\end{eqnarray*}
This together with equality~\eqref{eq:We1} implies
\begin{eqnarray}\label{eq:We2}
&&\||x'|^\gamma F(\sqrt{L})f\|_{L^2(\R_{x'}^{d_1}\times \R_{x''}^{d_2})}^2\nonumber\\
&& \leq  \sum_{k=0}^\infty\||\xi|^{-\gamma} ((2k+d_1)|\xi|)^{\gamma/2} G((2k+d_1)|\xi|)\sum_{|\nu|=k}\langle g, \Phi_\nu^\xi\rangle \Phi_\nu^\xi(x')\widehat{h}(\xi)\|_{L^2(\R_{x'}^{d_1}\times \R_{\xi}^{d_2})}^2.
\end{eqnarray}

Let $ \widetilde{G}_{k,x'}(|\xi|) $ be the  function on $\R$ defined by
$$
\widetilde{G}_{k,x'}(|\xi|)=|\xi|^{-\gamma} ((2k+d_1)|\xi|)^{\gamma/2} G((2k+d_1)|\xi|)\sum_{|\nu|=k}\langle g, \Phi_\nu^\xi\rangle \Phi_\nu^\xi(x').
$$
By estimate~\eqref{eq:We2} and Plancherel equality,
\begin{eqnarray}\label{eq:We3}
\||x'|^\gamma F(\sqrt{L})f\|_{L^2(\R_{x'}^{d_1}\times \R_{x''}^{d_2})}^2
&\leq& \sum_{k=0}^\infty \|\mathcal{F}^{-1}\left(\widetilde{G}_{k,x'}(|\xi|) \widehat{h}(\xi)\right)(x'')\|_{L^2(\R_{x'}^{d_1}\times \R_{x''}^{d_2})}^2\nonumber\\
&=&\sum_{k=0}^\infty \|\widetilde{G}_{k,x'}(\sqrt{-\Delta_2})h(x'')\|_{L^2(\R_{x'}^{d_1}\times \R_{x''}^{d_2})}^2.
\end{eqnarray}

Note that $\support G\subset [R^2/16,R^2]$. Thus $\support \widetilde{G}_{k,x'} \subset [0,R^2/(2k+d_1)]$.
Set  $m=2k+d_1$ and $a=R^2/(2k+d_1)$. By restriction type estimates for $-\Delta_2$ (see e.g. \cite{COSY}),
$$
\|\widetilde{G}_{k,x'}(\sqrt{-\Delta_2})h(x'')\|_{L^2(\R_{x''}^{d_2})}^2 \leq C a^{2d_2(\frac{1}{p}-\frac{1}{2})}\|\delta_{a}\widetilde{G}_{k,x'}\|^2_{L^2(\R)} \|h\|^2_{L^p(\R^{d_2})}.
$$
Thus by \eqref{eq:We3},
\begin{eqnarray}\label{eq:We4}
\||x'|^\gamma F(\sqrt{L})f\|_{L^2(\R_{x'}^{d_1}\times \R_{x''}^{d_2})}^2
&\leq & C\sum_{m=1}^\infty a^{2d_2(\frac{1}{p}-\frac{1}{2})}
\|\delta_{a}\widetilde{G}_{k,x'}\|^2_{L^2(\R_{x'}^{d_1}\times \R_{|\xi|} )} \|h\|^2_{L^p(\R^{d_2})}.
\end{eqnarray}

By Proposition~\ref{prop:REforLxi},
\begin{eqnarray}\label{eq:We5}
\|\delta_{a}\widetilde{G}_{k,x'}\|^2_{L^2(\R_{x'}^{d_1}\times \R_{|\xi|} )} & \leq & \int_{\R_{|\xi|}} |a\xi|^{-2\gamma} (am|\xi|)^{\gamma} |G(am|\xi|)|^2\|\sum_{|\nu|=k}\langle g, \Phi_\nu^{a\xi}\rangle \Phi_\nu^{a\xi}\|^2_{L^2(\R_{x'}^{d_1})}      \,d|\xi|\nonumber\\
&\leq&C \int_{\R_{|\xi|}} |a\xi|^{-2\gamma} (am|\xi|)^{\gamma} |G(am|\xi|)|^2 |a\xi|^{d_1(\frac{1}{p}-\frac{1}{2})}m^{d_1(\frac{1}{p}-\frac{1}{2})-1}\|g\|^2_{L^p(\R^{d_1})}      \,d|\xi|\nonumber\\
&\leq&C m^{2\gamma-1}\int_{\R} t^{-2\gamma} t^{\gamma} |G(t)|^2 t^{d_1(\frac{1}{p}-\frac{1}{2})}      \frac{dt}{ma}\|g\|^2_{L^p(\R^{d_1})}\nonumber\\
&\leq &C m^{2\gamma-1} R^{2d_1(\frac{1}{p}-\frac{1}{2})-2\gamma}\|\delta_{R^2}G\|_{L^2(\R)}^2\|g\|^2_{L^p(\R^{d_1})}.
\end{eqnarray}

Combing estimates~\eqref{eq:We4} and \eqref{eq:We5} and noting that
$\gamma<d_2(1/p-1/2)$ yields
\begin{eqnarray*}
&&\||x'|^\gamma F(\sqrt{L})f\|_{L^2(\R_{x'}^{d_1}\times \R_{x''}^{d_2})}^2\\
&&\leq  C \sum_{m=1}^\infty R^{4d_2(\frac{1}{p}-\frac{1}{2})+2d_1(\frac{1}{p}-\frac{1}{2})-2\gamma}
m^{2\gamma-2d_2(\frac{1}{p}-\frac{1}{2})-1}\|\delta_{R^2}G\|_{L^2(\R)}^2\|g\|^2_{L^p(\R^{d_1})}\|h\|^2_{L^p(\R^{d_2})}\\
&&\leq  C R^{4d_2(\frac{1}{p}-\frac{1}{2})+2d_1(\frac{1}{p}-\frac{1}{2})-2\gamma} \|\delta_{R}F\|_{L^2(\R)}^2 \|f\|^2_{L^p(\R^{d_1}\times \R^{d_2})}.
\end{eqnarray*}
This proves the estimate~\eqref{eq:WeightedRE}.

Next we prove  \eqref{eq:WeightedRE2}.
Similarly to the above derivation, we have
\begin{eqnarray}
&&\||x'|^\gamma F(\sqrt{L})P_{B(y,r)}f\|_{L^2(\R_{x'}^{d_1}\times \R_{x''}^{d_2})}^2\nonumber\\
&&\leq C\sum_{m=1}^\infty a^{2d_2(\frac{1}{p}-\frac{1}{2})}
\|\delta_{a}\widetilde{G}_{k,x'}\|^2_{L^2(\R_{x'}^{d_1}\times \R_{|\xi|} )} \|h\|^2_{L^p(\R^{d_2})}\nonumber\\
&&\leq C\sum_{m=1}^\infty a^{2d_2(\frac{1}{p}-\frac{1}{2})}
\int_{\R_{|\xi|}} |a\xi|^{-2\gamma} (am|\xi|)^{\gamma} |G(am|\xi|)|^2\|\sum_{|\nu|=k}\langle g, \Phi_\nu^{a\xi}\rangle \Phi_\nu^{a\xi}\|^2_{L^2(\R_{x'}^{d_1})}      \,d|\xi| \|h\|^2_{L^p(\R^{d_2})}\nonumber\\
&&\leq C \sum_{m=1}^\infty a^{2d_2(\frac{1}{p}-\frac{1}{2})}
\int_{\R} m^{2\gamma}t^{-\gamma} |G(t)|^2\|\sum_{|\nu|=k}\langle g, \Phi_\nu^{t/m}\rangle \Phi_\nu^{t/m}\|^2_{L^2(\R_{x'}^{d_1})}      \,dt/(am) \|h\|^2_{L^p(\R^{d_2})}\nonumber\\
&&\leq C\int_{\R}R^{4d_2(\frac{1}{p}-\frac{1}{2})-2}t^{-\gamma}|G(t)|^2\sum_{m=1}^\infty
 m^{2\gamma-2d_2(\frac{1}{p}-\frac{1}{2})} \|\sum_{|\nu|=k}\langle g, \Phi_\nu^{t/m}\rangle \Phi_\nu^{t/m}\|^2_{L^2(\R_{x'}^{d_1})}      \,dt \|h\|^2_{L^p(\R^{d_2})}\nonumber,
\end{eqnarray}
where the function $g$ has compact support such that $\support g\subset B(y',r)$ which is the standard ball defining by Euclidean distance in $\R^{d_1}$. Note that $\support G\subset [R^2/16,R^2]$. Thus
$R^2\sim t$ in the last integral and then
\begin{eqnarray}\label{eq:y>4r1}
&& \hspace{1cm} \||x'|^\gamma F(\sqrt{L})P_{B(y,r)}f\|_{L^2(\R_{x'}^{d_1}\times \R_{x''}^{d_2})}^2\\
&\leq&C\int_{\R}R^{2d_2(\frac{1}{p}-\frac{1}{2})-2}t^{d_2(\frac{1}{p}-\frac{1}{2})-\gamma}|G(t)|^2\sum_{m=1}^\infty
 m^{2\gamma-2d_2(\frac{1}{p}-\frac{1}{2})} \|\sum_{|\nu|=k}\langle g, \Phi_\nu^{t/m}\rangle \Phi_\nu^{t/m}\|^2_{L^2(\R_{x'}^{d_1})}      \,dt \|h\|^2_{L^p(\R^{d_2})}\nonumber.
\end{eqnarray}
Next we claim that for $0\leq \gamma<d_2(\frac{1}{p}-\frac{1}{2})$,
\begin{eqnarray}\label{eq:y>4r2}
&& \sum_{m=1}^\infty |y'|^{2d_2(\frac{1}{p}-\frac{1}{2})-2\gamma}t^{d_2(\frac{1}{p}-\frac{1}{2})-d_1(\frac{1}{p}-\frac{1}{2})-\gamma}m^{2\gamma-2d_2(\frac{1}{p}-\frac{1}{2})}\|\sum_{|\nu|=k}\langle g, \Phi_\nu^{t/m}\rangle \Phi_\nu^{t/m}\|^2_{L^2(\R_{x'}^{d_1})}\\
 &&\leq C\|g\|_{L^p}^2, \nonumber
%\hspace{-1.1cm}
\end{eqnarray}
where $C$ is independent of $t$ and $y'$.

In order to prove \eqref{eq:y>4r2} we split the sum into two parts: $m\leq \sqrt{t}|y'|/4$ and $m> \sqrt{t}|y'|/4$.

If $m> \sqrt{t}|y'|/4$, by Proposition~\ref{prop:REforLxi},
\begin{eqnarray*}
&&\sum_{m> \sqrt{t}|y'|/4}^\infty |y'|^{2d_2(\frac{1}{p}-\frac{1}{2})-2\gamma}t^{d_2(\frac{1}{p}-\frac{1}{2})-d_1(\frac{1}{p}-\frac{1}{2})-\gamma}m^{2\gamma-2d_2(\frac{1}{p}-\frac{1}{2})}\|\sum_{|\nu|=k}\langle g, \Phi_\nu^{t/m}\rangle \Phi_\nu^{t/m}\|^2_{L^2(\R_{x'}^{d_1})} \\
&&\leq C\sum_{m> \sqrt{t}|y'|/4}^\infty |y'|^{2d_2(\frac{1}{p}-\frac{1}{2})-2\gamma}t^{d_2(\frac{1}{p}-\frac{1}{2})-d_1(\frac{1}{p}-\frac{1}{2})-\gamma}m^{2\gamma-2d_2(\frac{1}{p}-\frac{1}{2})} (t/m)^{d_1(\frac{1}{p}-\frac{1}{2})} m^{d_1(\frac{1}{p}-\frac{1}{2})-1}\|g\|_{L^p}^2 \\
&&\leq  C \sum_{m> \sqrt{t}|y'|/4}^\infty (\sqrt{t}|y'|)^{2d_2(\frac{1}{p}-\frac{1}{2})-2\gamma}m^{2\gamma-2d_2(\frac{1}{p}-\frac{1}{2})-1}
\|g\|_{L^p}^2\\
&&\leq C \|g\|_{L^p}^2.
\end{eqnarray*}

If $m\leq \sqrt{t}|y'|/4$ and $x'\in B(y',r)$, then $|x'|\geq |y'|/2$ and $m\leq \sqrt{t}|x'|/2$. Moreover, $|m^{-1/2}\sqrt{t}x'|^2\geq 4m$. By \cite[Lemma 8]{MS}, we know that $\sum_{|\nu|=k}|\Phi_{\nu}(x')|^2\leq C\exp(-c|x'|^2)$ when $|x'|^2\geq 2(2k+d_1)$. Hence
\begin{eqnarray*}
\sum_{|\nu|=k}|\Phi_\nu^{t/m}(x')|^2 &=& |t/m|^{d_1/2}\sum_{|\nu|=k}|\Phi_\nu(m^{-1/2}\sqrt{t}x')|^2\\
&\leq&  C|t/m|^{d_1/2} e^{-ct|x'|^2/m}\\
&\leq&  C|t/m|^{d_1/2} e^{-ct|y'|^2/m}.
\end{eqnarray*}
Therefore,
\begin{eqnarray*}
\|\sum_{|\nu|=k}\langle g, \Phi_\nu^{t/m}\rangle \Phi_\nu^{t/m}\|^2_{L^2(\R_{x'}^{d_1})}&\leq& \|g\|_{L^p}^2 \sum_{|\nu|=k}\|\Phi_\nu^{t/m}\|_{L^{p'}(B(y',r))}^2\\
&\leq&\|g\|_{L^p}^2 \|(\sum_{|\nu|=k}|\Phi_\nu^{t/m}|^2)^{1/2}\|_{L^{p'}(B(y',r))}^2\\
&\leq& C \|g\|_{L^p}^2 |t/m|^{d_1/2} e^{-ct|y'|^2/m}r^{2d_1(1-\frac{1}{p})}.
\end{eqnarray*}
Hence
\begin{eqnarray*}
&&\sum_{m\leq \sqrt{t}|y'|/4}^\infty |y'|^{2d_2(\frac{1}{p}-\frac{1}{2})-2\gamma}t^{d_2(\frac{1}{p}-\frac{1}{2})-d_1(\frac{1}{p}-\frac{1}{2})-\gamma}m^{2\gamma-2d_2(\frac{1}{p}-\frac{1}{2})}\|\sum_{|\nu|=k}\langle g, \Phi_\nu^{t/m}\rangle \Phi_\nu^{t/m}\|^2_{L^2(\R_{x'}^{d_1})} \\
&&\leq \sum_{m\leq \sqrt{t}|y'|/4}^\infty (\sqrt{t}|y'|)^{2d_1(1-\frac{1}{p})+2d_2(\frac{1}{p}-\frac{1}{2})-2\gamma}m^{-2d_2(\frac{1}{p}-\frac{1}{2})-d_1/2}e^{-ct|y'|^2/m}
\|g\|_{L^p}^2\\
&&\leq  \sum_{m=1}^\infty \sup_{u>m}u^{d_1(1-\frac{1}{p})+d_2(\frac{1}{p}-\frac{1}{2})-\gamma}m^{d_1(1-\frac{1}{p})-d_1/2}e^{-cu}\|g\|_{L^p}^2\\
&&\leq C \|g\|_{L^p}^2,
\end{eqnarray*}
which complete the proof of the claim~\eqref{eq:y>4r2}.

Combining \eqref{eq:y>4r2} and  \eqref{eq:y>4r1}, we obtain
\begin{eqnarray*}
&&\||x'|^\gamma F(\sqrt{L})P_{B(y,r)}f\|_{L^2(\R_{x'}^{d_1}\times \R_{x''}^{d_2})}^2\nonumber\\
&&\leq  C\int_{\R}R^{2d_2(\frac{1}{p}-\frac{1}{2})-2}|y'|^{2\gamma-2d_2(\frac{1}{p}-\frac{1}{2})}t^{d_1(\frac{1}{p}-\frac{1}{2})}|G(t)|^2     \,dt  \|g\|_{L^p}^2 \|h\|^2_{L^p(\R^{d_2})}\\
&&\leq C R^{2(d_2+d_1)(\frac{1}{p}-\frac{1}{2})}|y'|^{2\gamma-2d_2(\frac{1}{p}-\frac{1}{2})}\int_{\R}|G(t)|^2     \,dt/R^2  \|g\|_{L^p}^2 \|h\|^2_{L^p(\R^{d_2})}\\
&&\leq  C R^{2(d_2+d_1)(\frac{1}{p}-\frac{1}{2})}|y'|^{2\gamma-2d_2(\frac{1}{p}-\frac{1}{2})}\|\delta_R F\|^2_{L^2} \|f\|^2_{L^p(\R^{d_1}\times\R^{d_2})}.
\end{eqnarray*}
This completes the proof of estimate~\eqref{eq:WeightedRE2} and so the proof of Theorem~\ref{thm:WeightedRE}.
\end{proof}

\begin{remark}
Under the assumptions of Theorem~\ref{thm:WeightedRE}, when $\gamma=0$, the estimate~\eqref{eq:WeightedRE} holds for all $1\leq p\leq (2d_2+2)/(d_2+3)$, which means that the condition $p<2d_1/(d_1+2)$ is not necessary in this case. Actually, in our proof, if $\gamma=0$, we do not  need the sharp order $\frac{d_1}{2}(\frac{1}{p}-\frac{1}{2})-\frac{1}{2}$ for $2k+d_1$ in the estimate~\eqref{eq:ReforLxi}. We only need that for all $1\leq p\leq 2$
\begin{eqnarray*}
\|\sum_{|\nu|=k}\langle f,\Phi_\nu^\xi \rangle \Phi_\nu^\xi\|_{L^2(\R^{d_1})}\leq C |\xi|^{\frac{d_1}{2}(\frac{1}{p}-\frac{1}{2})}(2k+d_1)^{(\frac{d_1}{2}-1)(\frac{1}{p}-\frac{1}{2})}\|f\|_{L^p(\R^{d_1})},
\end{eqnarray*}
which can be achieved by interpolation between $p=1$ and the fact that
$$
\|\sum_{|\nu|=k}\langle f,\Phi_\nu^\xi \rangle \Phi_\nu^\xi\|_{L^2}\leq C\|f\|_{L^2}.
$$
See also \cite{LS}.
\end{remark}

\bigskip

\section{Spectral multipliers for compactly supported functions}
\label{sec:H1toL1} \setcounter{equation}{0}

As mentioned in Section 2, the heat kernel of the operator $L$ satisfies a Gaussian upper bound given in terms of the distance
$\rho$.  In addition, $L$ satisfies the Davies-Gaffney estimate and  the finite speed propagation property, see \cite{RS}.
On the other hand, we proved restriction type estimates for the operator $L$ in Section 3. Therefore we may follow ideas in 
\cite{COSY}, Sections 3 and 4,  to prove  spectral multiplier results as well as  Bochner-Riesz summability  results.

Define the  multiplication  operator $w_\gamma$ on $\R^{d_1}\times \R^{d_2}$ by
$$
w_\gamma f(x',x'')=|x'|^{\gamma}f(x',x'').
$$
\begin{lemma}\label{lem:pto2toptop}
Let $F : [0, \infty) \to \mathbb{C}$ be a bounded Borel function. We denote by $K_{F(L)}$ the Schwartz kernel of $F(L)$. Assume that
$$
\support K_{F(L)} \subset \mathcal{D}_r=\{(x,y)\in (\R^{d_1}\times \R^{d_2})\times (\R^{d_1}\times \R^{d_2}): \rho(x,y)\leq r\}
$$
for some $r>0$. Then for $1\leq p\leq 2$, there exists a constant $C=C_{p}$ such that for $\gamma\in [0,d_1(1/p-1/2))$
\begin{eqnarray*}
\|F(L)\|_{p\to p}&\leq& C\sup_{|y'|\leq 4r}
\{r^{(2d_2+d_1)(\frac{1}{p}-\frac{1}{2})-\gamma}\|w_\gamma F(L)P_{B(y,r)}\|_{p\to 2}\}\\
&&+\ C\sup_{|y'|> 4r}
\{r^{(d_2+d_1)(\frac{1}{p}-\frac{1}{2})}|y'|^{d_2(\frac{1}{p}-\frac{1}{2})-\gamma}\|w_\gamma F(L)P_{B(y,r)}\|_{p\to 2}\}.
\end{eqnarray*}

\end{lemma}

\begin{proof}
First we choose a sequence $(x_n)  \in \R^{d_1}\times \R^{d_2}$ such that
$\rho(x_i,x_j)> r/10$ for $i\neq j$ and $\sup_{x\in \R^{d_1}\times \R^{d_2} }\inf_i \rho(x,x_i)
\le r/10$. Such sequence exists because $\R^{d_1}\times \R^{d_2}$ is separable under the new distance $\rho$.
Second we let $B_i=B(x_i, r)$ and define $\widetilde{B_i}$ by the formula
$$\widetilde{B_i}=\bar{B}\left(x_i,\frac{r}{10}\right)\setminus
\bigcup_{j<i}\bar{B}\left(x_j,\frac{r}{10}\right),$$
where $\bar{B}\left(x, r\right)=\{y\in \R^{d_1}\times \R^{d_2} \colon \rho(x,y)
\le r\}$. Third we put $\chi_i=\chi_{\widetilde B_i}$, where
$\chi_{\widetilde B_i}$ is the characteristic function of the set
${\widetilde B_i}$. Note that for $i\neq j$,
 $B(x_i, \frac{r}{20}) \cap B(x_j, \frac{r}{20})=\emptyset$. Hence
$$
K :=\sup_i\#\{j:\;\rho(x_i,x_j)\le  2r\} \le
  \sup_x  \frac{|B(x, (2+\frac{1}{20})r)|}{
  |B(x, \frac{r}{20})|}< C  41^{d_1+2d_2}< \infty.
$$
It is not difficult to see that
$$
\D_{r}  \subset \bigcup_{\{i,j:\, \rho(x_i,x_j)<
 2 r\}} \widetilde{B}_i\times \widetilde{B}_j \subset \D_{4 r}.
$$
Therefore,
$$
F(L)f = \sum_{i,j:\, {\rho}(x_i,x_j)< 2r} P_{\widetilde B_i}F(L) P_{\widetilde
B_j}f.
$$
Hence by H\"older's inequality
\begin{eqnarray}\label{eq:pto2toptop}
\|F(L) f\|_{p}^p &=& \|\sum_{i,j:\, {\rho}(x_i,x_j)< 2r} P_{\widetilde B_i}F(L)
P_{\widetilde B_j}f\|_{p}^p \nonumber\\
& =&\sum_i \|\sum_{j:\,{\rho}(x_i,x_j)<
2r} P_{\widetilde B_i}F(L)P_{\widetilde B_j}f\|_{p}^p  \nonumber \\
&\le& CK^{p-1}  \sum_i \sum_{j:\,{\rho}(x_i,x_j)< 2r} \| P_{\widetilde
B_i}F(L)P_{\widetilde B_j}f\|_{p}^p \nonumber\\
&\le& CK^{p-1}  \sum_i   \sum_{j:\,{\rho}(x_i,x_j)< 2r}
\||x'|^{-\gamma}\|^p_{L^q(\widetilde{B}_i)} \||x'|^\gamma P_{\widetilde
B_i}TP_{\widetilde B_j}f\|_{2}^p
\end{eqnarray}
where $1/q=1/p-1/2$.

Now we estimate $\||x'|^{-\gamma}\|^p_{L^q(\widetilde{B}_i)}$.

Suppose  $|x_j'|>4r$. Since  $|x_i'-x_j'|<\rho(x_i,x_j)<2r$, we have
$|x_i'|>2r$ and $2|x_j'|>|x_i'|>|x_j'|/2$. Thus, for  $x\in \widetilde{B}_i$, that is, $\rho(x,x_i)\leq r/10$, we have $|x'-x_i'|\leq r/10$. This implies $|x'|>r$ and $|x'|>|x_i'|/2>|x_j'|/4$. Then by \eqref{eqVr}
\begin{eqnarray}\label{eq:volE1}
\||x'|^{-\gamma}\|^p_{L^q(\widetilde{B}_i)}\leq C|x_j'|^{-p\gamma}
\mu(\widetilde{B}_i)^{p/q}\leq C|x_j'|^{-p\gamma}r^{p(d_1+d_2)(\frac{1}{p}-\frac{1}{2})}|x_j'|^{pd_2(\frac{1}{p}-\frac{1}{2})}.
\end{eqnarray}

If  $|x_j'|\leq 4r$, $|x_i'|\leq |x_i'-x_j'|+|x_j'|\leq 6r$. So $x\in \widetilde{B}_i$ implies $|x'|\leq 7r$ and $|x''-x_i''|\leq 13r^2$.
Then
\begin{eqnarray}\label{eq:volE2}
\||x'|^{-\gamma}\|^q_{L^q(\widetilde{B}_i)}\leq \int_{|x''-x_i''|\leq 13r^2}
\int_{|x'|\leq 7r} |x'|^{-q\gamma}dx'dx''\leq C r^{2d_2+d_1-q\gamma}.
\end{eqnarray}

Substituting estimates~\eqref{eq:volE1} and \eqref{eq:volE2} in
\eqref{eq:pto2toptop} finishes  the proof of Lemma~\ref{lem:pto2toptop}.

\end{proof}

Now  we can state and prove the following multiplier theorem for compactly  supported functions.
Recall that $D = \max(d_1 + d_2, 2 d_2)$.
\begin{theorem}\label{th:compactmultipliers}
Suppose that a bounded Borel function $F : \R \to \mathbb{C}$ with compact support in $[1/4,1]$ satisfies
\begin{equation*}
\| F \|_{W_2^s}  < \infty
\end{equation*}
for some $s > D|1/p-1/2|$.  Then the  operator $F(tL)$ is  bounded on
 $L^p(\R^{d_1}\times\R^{d_2})$.  In addition
\begin{equation*}
 \sup_{t>0}\|F(tL)\|_{L^p \to L^{p}} \leq C_p \|F\|_{W_2^s}.
\end{equation*}

\end{theorem}

\begin{proof}
Let   $\eta \in C_c^{\infty}(\mathbb R) $   be even and  such that $\support \eta\subseteq \{ \xi: 1/4\leq |\xi|\leq 1\}$ and
$$
\sum_{\ell\in \ZZ} \eta(2^{-\ell} \lambda)=1 \ \ \quad \forall
{\lambda>0}.
$$
Then we  set $\eta_0(\lambda)= 1-\sum_{\ell> 0} \eta(2^{-\ell} \lambda)$,
\begin{eqnarray}\label{eq3.7}
F^{(0)}(\lambda)=\frac{1}{2\pi}\int_{-\infty}^{+\infty}
 \eta_0(t) \hat{F}(t) \cos(t\lambda) \;dt
\end{eqnarray}
and
\begin{eqnarray}\label{eq3.8}
F^{(\ell)}(\lambda) =\frac{1}{2\pi}\int_{-\infty}^{+\infty}
 \eta(2^{-\ell}t) \hat{F}(t) \cos(t\lambda) \;dt.
\end{eqnarray}
Note that in virtue of the Fourier inversion formula
$$
F(\lambda)=\sum_{\ell \ge 0}F^{(\ell)}(\lambda)
$$
and by \cite[Lemma 2.1]{COSY},
$$
\support K_{F^{(\ell)}(t\SL)} \subset \D_{2^{\ell}t}.
$$
Now by Lemma~\ref{lem:pto2toptop}
\begin{eqnarray}\label{eq3.9}
\big\|F(t\SL)\big\|_{p\to p } &\le& \sum_{\ell \ge 0}\big\|F^{(\ell)}(t\sqrt {L}) \big\|_{p\to p }\nonumber\\
&\le& C \sum_{\ell \ge 0}\sup_{|y'|\leq 42^\ell t}
\{(2^\ell t)^{(2d_2+d_1)(\frac{1}{p}-\frac{1}{2})-\gamma}\|w_\gamma F^{(\ell)}(t\sqrt {L})P_{B(y,2^\ell t)}\|_{p\to 2}\}\nonumber\\
&&+\ C \sum_{\ell \ge 0}\sup_{|y'|> 42^\ell t}
\{(2^\ell t)^{(d_2+d_1)(\frac{1}{p}-\frac{1}{2})}|y'|^{d_2(\frac{1}{p}-\frac{1}{2})-\gamma}\|w_\gamma F^{(\ell)}(t\sqrt {L})P_{B(y,2^\ell t)}\|_{p\to 2}\}.
\end{eqnarray}
Since $F^{(\ell)}$ is not compactly
supported we choose a function $\psi \in C_c^\infty(1/16, 4)$ such that $\psi(\lambda)=1$ for $\lambda \in (1/8,2)$
and note that
\begin{eqnarray} \label{eq3.10}
&& \big\|w_\gamma F^{(\ell)}(t\sqrt {L})P_{B(y,2^{\ell}t)}\big\|_{p\to 2} \nonumber\\
&& \le  \big\|w_\gamma\big(\psi
F^{(\ell)}\big)(t\sqrt {L})P_{B(y,2^{\ell}t)}\big\|_{p\to 2 }
 +\big\|w_\gamma\big((1-\psi)F^{(\ell)}\big)(t\sqrt {L})P_{B(y,2^{\ell}t)}\big\|_{p\to
2 }.
\end{eqnarray}

To estimate  the norm $\|w_\gamma\big(\psi
F^{(\ell)}\big)(t\sqrt {L})P_{B(y,2^{\ell}t)}\|_{p\to 2 }$, we use the
weighted restriction  estimates~\eqref{eq:WeightedRE} and the fact that $\psi \in
C_c(1/16,4)$  to obtain
$$
\big\|w_\gamma\big(\psi F^{(\ell)}\big)(t\sqrt {L})P_{B(y,2^{\ell}t)}\big\|_{p\to 2 }
\le C t^{-(2d_2+d_1)(1/p-1/2)+\gamma} \big\| \delta_{t^{-1}}\big(\psi
F^{(\ell)}\big) (t\cdot)\big\|_{L^2}
$$
and for  $|y'|\geq 2^{\ell+2} t$
$$
\big\|w_\gamma\big(\psi F^{(\ell)}\big)(t\sqrt
{L})P_{B(y,2^{\ell}t)}\big\|_{p\to 2 } \le C
t^{-(d_2+d_1)(1/p-1/2)}|y'|^{\gamma-d_2(1/p-1/2)} \big\|
\delta_{t^{-1}}\big(\psi F^{(\ell)}\big) (t\cdot)\big\|_{L^2}
$$ for all $t>0$.

 If  $|y'|>2^{\ell+2} t$, it follows from $s>(d_1+d_2)(1/p-1/2)$ that
 \begin{eqnarray}\label{eq3.11}
&&\sum_{\ell \ge 0}\sup_{|y'|> 2^{\ell+2} t}
\{(2^\ell t)^{(d_2+d_1)(\frac{1}{p}-\frac{1}{2})}|y'|^{d_2(\frac{1}{p}-\frac{1}{2})-\gamma}\|w_\gamma (\psi F^{(\ell)})(t\sqrt {L})P_{B(y,2^\ell t)}\|_{p\to 2}\}\nonumber\\
 &&\le C\sum_{\ell \ge 0}   2^{\ell (d_1+d_2)(\frac{1}{p}-\frac{1}{2})}\big\|  \delta_{t^{-1}}
\big(\psi F^{(\ell)}\big) (t\cdot)\big\|_{L^2}\nonumber\\
&&\leq C \sum_{\ell \ge 0}   2^{\ell (d_1+d_2)(\frac{1}{p}-\frac{1}{2})}\| F^{(\ell)} \|_{L^2}
\\
 &&\leq C\|F\|_{W_{2}^s}.\nonumber
\end{eqnarray}

For $|y'|\leq 2^{\ell+2} t$, we take $\gamma<\min\{d_1,d_2\}(1/p-1/2)$ such that $\min\{d_1,d_2\}(1/p-1/2)-\gamma$ is small enough and $s-D(1/p-1/2)>\min\{d_1,d_2\}(1/p-1/2)-\gamma$. Then for $s>D(1/p-1/2)$
 \begin{eqnarray}\label{eq3.111}
&&\sum_{\ell \ge 0}\sup_{|y'|\leq 2^{\ell+2} t}
\{(2^\ell t)^{(2d_2+d_1)(\frac{1}{p}-\frac{1}{2})-\gamma}\|w_\gamma (\psi F^{(\ell)})(t\sqrt {L})P_{B(y,2^\ell t)}\|_{p\to 2}\}\nonumber\\
& &\le C\sum_{\ell \ge 0}   2^{\ell \left(D(\frac{1}{p}-\frac{1}{2})+\min\{d_1,d_2\}(1/p-1/2)-\gamma\right)}\big\|  \delta_{t^{-1}}
\big(\psi F^{(\ell)}\big) (t\cdot)\big\|_{L^2}\nonumber\\
&&\leq C \sum_{\ell \ge 0}   2^{\ell \left(D(\frac{1}{p}-\frac{1}{2})+\min\{d_1,d_2\}(1/p-1/2)-\gamma\right)}\| F^{(\ell)} \|_{L^2}
\\
 &&\leq C\|F\|_{W_{2}^s}.\nonumber
\end{eqnarray}

Next we show bounds for
   $\big\|w_\gamma\big((1-\psi)F^{(\ell)}\big)(t\sqrt
{L})P_{B(y,2^{\ell}t)}\big\|_{p\to 2 }$.
Since the function $1-\psi$ is supported outside the
interval $(1/8,2)$, we can choose  a function $\phi\in C_c^{\infty}(2,8)$  such that
$$
1=\psi(\l)+\sum_{k\geq 0}\phi(2^{-k}\l)+\sum_{k\leq -6}\phi(2^{-k}\l)=\psi(\l)+\sum_{k\geq
0}\phi_{k}(\l)+\sum_{k\leq -6}\phi_{k}(\l)\quad \quad \forall \l>0.
$$
Hence
$$
\big((1-\psi)F^{(\ell)}\big)(\lambda)=  (\sum_{k\geq 0}+\sum_{k\leq -6})
 \big(\phi_{k}F^{(\ell)}\big)(\lambda) \quad \quad \forall \l>0.
$$

Note that  by the Gaussian upper bound for the heat kernel of $L$, we have  $E_{\SL}\{0\}=0$.
So it follows from Theorem~\ref{thm:WeightedRE} that
\begin{eqnarray*}
&&\big\|w_\gamma\big((1-\psi)F^{(\ell)}\big)(t\sqrt {L})P_{B(y,2^{\ell}t)}\big\|_{p\to 2 }\\
 &&\le   (\sum_{k\geq 0}+\sum_{k\leq -6})
\big\|w_\gamma\big(\phi_{k}F^{(\ell)}\big)(t\sqrt {L})P_{B(y,2^{\ell}t)}\big\|_{p\to 2}
\\
 &&\le   C (\sum_{k\geq 0}+\sum_{k\leq -6})(2^kt^{-1})^{(2d_2+d_1)(1/p-1/2)-\gamma}
 \big\|\delta_{ {2^{k+3}t^{-1}} }\big(\phi_kF^{(\ell)})(t \cdot)\big\|_{\infty}.
\end{eqnarray*}
Note that
$\supp F \subset [1/4, 1]$,  $\support \phi\subset [2, 8]$ and $\check{\eta}$ is in the Schwartz class so
\begin{eqnarray*}
 \big\|\phi_kF^{(\ell)}\big\|_{\infty}= 2^{\ell}\big\|\phi_k(F * \delta_{2^{\ell}}\check{\eta} )   \big\|_{\infty}
\leq C 2^{-M(\ell+\max\{0,k\})}\|F\|_{L^2}
\end{eqnarray*}
and similarly, $ \big\|\phi_kF^{(0)}\big\|_{\infty}\leq C 2^{-M\max\{0,k\}}\|F\|_{L^2}$.
 Therefore
\begin{eqnarray*}
 \big\|w_\gamma\big((1-\psi)F^{(\ell)}\big)(t\sqrt {L})P_{B(y,2^{\ell}t)}\big\|_{p\to 2 }
 &\leq& C  2^{-M\ell} t^{-(2d_2+d_1)(1/p-1/2)+\gamma}
 \|F\|_{L^2}
\end{eqnarray*}
and when $|y'|\geq 2^{\ell+2} t$
\begin{eqnarray*}
 \big\|w_\gamma\big((1-\psi)F^{(\ell)}\big)(t\sqrt {L})P_{B(y,2^{\ell}t)}\big\|_{p\to 2 }
 &\leq& C  2^{-M\ell} t^{-(d_2+d_1)(1/p-1/2)}|y'|^{\gamma-d_2(1/p-1/2)}
 \|F\|_{L^2}
\end{eqnarray*}
 Then by a similar calculation as
in~\eqref{eq3.11} and \eqref{eq3.111},
\begin{eqnarray}\label{eq3.14}
\sum_{\ell \ge 0}\sup_{|y'|> 2^{\ell+2} t}
\{(2^\ell t)^{(d_2+d_1)(\frac{1}{p}-\frac{1}{2})}|y'|^{d_2(\frac{1}{p}-\frac{1}{2})-\gamma}\|w_\gamma ((1-\psi) F^{(\ell)})(t\sqrt {L})P_{B(y,2^\ell t)}\|_{p\to 2}\}\leq  C \|F\|_{L^2 }.\hspace{-1.5cm}
\end{eqnarray}
and
\begin{eqnarray}\label{eq3.144}
\sum_{\ell \ge 0}\sup_{|y'|\leq 2^{\ell+2} t}
\{(2^\ell t)^{(2d_2+d_1)(\frac{1}{p}-\frac{1}{2})-\gamma}\|w_\gamma ((1-\psi) F^{(\ell)})(t\sqrt {L})P_{B(y,2^\ell t)}\|_{p\to 2}\}\leq  C \|F\|_{L^2 }.\hspace{-1.5cm}
\end{eqnarray}
Now we combine \eqref{eq3.9}, \eqref{eq3.10},  (\ref{eq3.11}), (\ref{eq3.111}), (\ref{eq3.14}) and
(\ref{eq3.144}) to  complete the proof of Theorem~\ref{th:compactmultipliers}.

\end{proof}

\section{Proofs of Theorems~\ref{thm1m} and \ref{thm2r}} \label{secProofs}

In this section, we prove our main results, i.e.,  Theorems~\ref{thm1m} and
\ref{thm2r}. The two results follow  from Theorem~\ref{th:compactmultipliers}. To do this, we
need a theorem from \cite{SYY} which states that singular  multiplier results follow  from
the corresponding one  for compactly supported functions.   We recall this explicitly. 

Let  $(X,\rho,\mu)$ be a metric
measure space satisfying doubling condition, that is, for all $r>0$
and $\lambda>1$, 
$$
\mu(B(x,\lambda r))\leq C\lambda^Q\mu(B(x,r)),
$$
where  $C$ and $Q$ are positive constants. 
Let $A$ be a non-negative self-adjoint operator which satisfies the
following two off-diagonal estimates: for some $m\geq 2$, some
$p_0\in [1,2]$ and for all $t>0$ and all $x,y\in X$
$$
\|P_{B(x,t^{1/m})}e^{-tA}P_{B(y,t^{1/m})}\|_{2\to 2}\leq
C\exp\left(-c\Big(\frac{\rho(x,y)}{t^{1/m}}\Big)^{\frac{m}{m-1}}\right)\leqno{(DG_m)}
$$
and
$$
\|P_{B(x,t^{1/m})}e^{-tA}P_{B(y,t^{1/m})}\|_{p_0\to 2}\leq
C\mu(B(x,t^{1/m}))^{-(\frac{1}{p_0}-\frac{1}{2})}\exp\left(-c\Big(\frac{\rho(x,y)}{t^{1/m}}\Big)^{\frac{m}{m-1}}\right).
\leqno{(G_{p_0,2,m})}
$$
Let again $\eta$ be a non trivial $C^\infty$ function with compact support in $(0, \infty)$. We have 
\begin{theorem}\label{th:compacttogeneral}
Let $A$ be a non-negative self-adjoint operator on $L^2(X)$
satisfying off-diagonal estimates~$(DG_m)$ and $(G_{p_0,2,m})$ for
some $1\leq p_0<2$. Assume that for any bounded Borel function $H$
such that $\support H\subset [1/4,4]$, the following condition
holds:
$$
\sup_{t>0}\|H(t\sqrt[m]{A})\|_{p\to p}\leq C\|H\|_{W_q^\alpha}
$$
for some $p\in (p_0,2)$, $\alpha>1/q$, and $1\leq q\leq \infty$.
Then for any bounded Borel function $F$ such that
$$
\sup_{t>0}\|\eta F(t\cdot)\|_{W_q^\alpha}<\infty,
$$
the operator $F(A)$ is bounded on $L^r(X)$ for all $p<r<p'$. In
addition,
$$
\|F(A)\|_{r\to r}\leq C\sup_{t>0}\|\eta F(t\cdot)\|_{W_q^\alpha}.
$$
\end{theorem}

This theorem is taken from  \cite[Theorem 3.3]{SYY}. It is   stated there with the additional assumption that
 $\alpha > Q (\frac{1}{p} - \frac{1}{2})$ where $Q$ is "the" homogeneous dimension. An inspection
of the proof shows that this condition is not needed and the theorem is valid  for $\alpha > \frac{1}{q}$ without appealing to
any dimension.

\bigskip

\noindent {\bf Proofs of Theorems~\ref{thm1m} and \ref{thm2r}}. Note
that from Theorem~\ref{th:Gaussian}, the Grushin operator satisfies
Gaussian upper bound and so it satisfies off-diagonal
estimates~$(DG_m)$ and $(G_{p_0,2,m})$ for $m=1$ and $p_0=1$. Then
Theorem~\ref{thm1m} follows from Theorem~\ref{th:compactmultipliers}
and \ref{th:compacttogeneral}.

To prove Theorem~\ref{thm2r}, we decompose the Bochner-Riesz means
$$
(1-tL)_+^\delta=\phi(L)(1-tL)_+^\delta+(1-\phi(L))(1-tL)_+^\delta,
$$
where $\phi$ is a smooth cutoff function on $\R$ with $\support \phi\subset [-1/2,1/2]$ and $\phi=1$ on interval $[-1/4,1/4]$. Then when $\delta > \max\{D|1/p-1/2|-1/2,0\}$,  $(1-\phi(L))(1-tL)_+^\delta$  is
uniformly bounded on $L^p$ by  Theorem~\ref{th:compactmultipliers}.
For $\phi(L)(1-tL)_+^\delta$, because the function $\phi(\lambda)(1-t\lambda)_+^\delta$ is  smooth  for all $\delta>0$, so the $L^p$-boundedness follows from the Gaussian bound of heat kernel of the operator $L$  and the spectral  multiplier result in \cite[Theorem 3.1]{DOS} or  \cite[Theorem 3.1]{COSY}.

\end{document}